\documentclass[11pt,reqno]{amsart}

\usepackage{amssymb}
\usepackage{amsfonts}

\advance \oddsidemargin -1.2in 
\advance \evensidemargin -1.2in 
\advance \textwidth 2in 
\advance \textheight 1.25in 
\advance \topmargin -0.5in

\newcommand{\T}{\ensuremath{\mathbb{T}}}
\newcommand{\R}{\ensuremath{\mathbb{R}}}
\newcommand{\N}{\ensuremath{\mathbb{N}}}
\newcommand{\Z}{\ensuremath{\mathbb{Z}}}

\theoremstyle{plain}
\newtheorem{theorem}{Theorem}
\newtheorem{cor}[theorem]{Corollary}
\newtheorem{lemma}[theorem]{Lemma}

\newtheorem{remark}[theorem]{Remark}
\newtheorem{example}[theorem]{Example}

\numberwithin{equation}{section}
\numberwithin{theorem}{section}

\small\normalsize
\begin{document}

\title[Young's Integral Inequality with Bounds]{Young's integral inequality with upper and lower bounds}

\author[Anderson]{Douglas~R.~Anderson}
\address{Concordia College, Department of Mathematics and Computer Science, Moorhead, MN 56562 USA}
\email{andersod@cord.edu}
\urladdr{http://www.cord.edu/faculty/andersod/} 

\author[Noren]{Steven Noren}
\address{Concordia College, Moorhead, MN 56562 USA}
\email{srnoren@cord.edu}

\author[Perreault]{Brent Perreault}
\address{Concordia College,  Moorhead, MN 56562 USA}
\email{bmperrea@cord.edu}


\begin{abstract} 
Young's integral inequality is reformulated with upper and lower bounds for the remainder. The new inequalities improve Young's integral inequality on all time scales, such that the case where equality holds becomes particularly transparent in this new presentation. The corresponding results for difference equations are given, and several examples are included. We extend these results to piecewise-monotone functions as well.
\end{abstract}

\maketitle\thispagestyle{empty}


\section{introduction}

In 1912, Young \cite{young} presented the following highly intuitive integral inequality, namely that any real-valued continuous function $f:[0,\infty)\rightarrow[0,\infty)$ satisfying $f(0)=0$ with $f$ strictly increasing on $[0,\infty)$ satisfies
\begin{equation}\label{yeq}
 ab \le \int_0^a f(t)dt + \int_0^b f^{-1}(y)dy
\end{equation}
for any $a,b\in[0,\infty)$, with equality if and only if $b=f(a)$. A useful consequence of this theorem is Young's inequality,
$$ ab \le \frac{a^p}{p}+\frac{b^q}{q}, \qquad \frac{1}{p}+\frac{1}{q}=1, $$
with equality if and only if $a^p=b^q$, a fact derived from \eqref{yeq} by taking $f(t)=t^{p-1}$ and $q=\frac{p}{p-1}$. Hardy, Littlewood, and P\'{o}lya included Young's inequality in their classic book \cite{hardy}, but there was no analytic proof until Diaz and Metcalf \cite{diaz} supplied one in 1970. Tolsted \cite{tolsted} showed how to derive Cauchy, H\"{o}lder, and Minkowski inequalities in a straightforward way from \eqref{yeq}. For many other applications and extensions of Young's inequality, see Mitrinovi\'{c}, Pe\v{c}ari\'{c}, and Fink \cite{fink}. For the purposes of this paper we recall some results that consider upper bounds for the integrals in \eqref{yeq}. Merkle \cite{merkle} established the inequality
$$  \int_0^a f(t)dt + \int_0^b f^{-1}(y)dy \le \max\{af(a),bf^{-1}(b)\}, $$ 
which has been improved and reformulated recently by Minguzzi \cite{em} to the inequality
\begin{equation}\label{emorig}
 0 \le \int_{\alpha_1}^a f(t)dt + \int_{\beta_1}^b f^{-1}(y)dy - ab + \alpha_1\beta_1 \le \left(f^{-1}(b)-a\right)\left(b-f(a)\right), 
\end{equation}
where the hypotheses of Young's integral inequality hold, except that $f(\alpha_1)=\beta_1$ has replaced $f(0)=0$.

One might wonder if there is a discrete version of \eqref{yeq} in the form of a summation inequality, or more generally a time-scale version of \eqref{yeq}, where a time scale, introduced by Hilger \cite{hilger}, is any nonempty closed set of real numbers. Wong, Yeh, Yu, and Hong \cite{wyyh} presented a version of Young's inequality on time scales $\T$ in the following form. Using the standard notation \cite{bp1} of the left jump operator $\rho$ given by $\rho(t):=\sup\{s\in\T: s<t\}$, the right jump operator $\sigma$ given by $\sigma(t)=\inf\{s\in\T: s>t\}$, the compositions $f\circ\rho$ and $f\circ\sigma$ denoted by $f^\rho$ and $f^\sigma$, respectively, the graininess functions defined by $\mu(t)=\sigma(t)-t$ and $\nu(t)=t-\rho(t)$, and the delta and nabla derivatives of $f$ at $t\in\T$, denoted $f^\Delta(t)$ and $f^\nabla(t)$, respectively, (provided they exist) are given by
$$ f^\Delta(t):=\lim_{s\rightarrow t}\frac{f^\sigma(t)-f(s)}{\sigma(t)-s}, \qquad  f^\nabla(t):=\lim_{s\rightarrow t}\frac{f^\rho(t)-f(s)}{\rho(t)-s}, $$
we have the following result.

\begin{theorem}[Wong, Yeh, Yu, and Hong]\label{amlthm}
Let $f$ be right-dense continuous on $[0,c]_\T:=[0,c]\cap\T$ for $c>0$, strictly increasing, with $f(0)=0$. Then for $a\in[0,c]_\T$ and $b\in[0,f(c)]_\T$ the inequality
$$ ab \le \int_{0}^{a}f^\sigma(t)\Delta t + \int_{0}^{b}(f^{-1})^\sigma(y)\Delta y $$
holds.
\end{theorem}
\noindent If $\T=\Z$ and $f(t)=t$, then Theorem \ref{amlthm} says that
\begin{equation}\label{wongexample}
 ab \le \sum_{t=0}^{a-1}(t+1) + \sum_{y=0}^{b-1}(y+1) = \frac{1}{2}\left(a(a+1)+b(b+1)\right)
\end{equation}
holds. Note that an if and only if clause concerning an actual equality is missing in the formulation in Theorem \ref{amlthm}, with equality impossible in the simple example \eqref{wongexample} except for the trivial case $a=0=b$. This omission was rectified in \cite{and} via the following theorem.


\begin{theorem}[Anderson]\label{yng2}
Let $\T$ be any time scale (unbounded above) with $0\in\T$. Further, suppose that $f:[0,\infty)_{\T}\rightarrow\R$ is a real-valued function satisfying
\begin{enumerate}
 \item $f(0)=0$;
 \item $f$ is continuous on $[0,\infty)_\T$, right-dense continuous at $0$;
 \item $f$ is strictly increasing on $[0,\infty)_\T$ such that $f(\T)$ is also a time scale.
\end{enumerate}
Then for any $a\in[0,\infty)_\T$ and $b\in[0,\infty)\cap f(\T)$, we have
$$ ab \le \frac{1}{2}\int_0^a \left[f(t)+f(\sigma(t))\right]\Delta t + \frac{1}{2}\int_0^b \left[f^{-1}(y)+f^{-1}(\sigma(y))\right]\Delta y, $$
with equality if and only if $b=f(a)$.
\end{theorem}

\noindent

Motivated by \cite{em}, in this paper we extend \eqref{emorig} to the general time scales setting while, in the process, simplifying and extending Theorem \ref{yng2} as well. As these results on time scales will include new results in difference equations as an important corollary, we will illustrate our new inequalities using discrete examples with $\T=\Z$.
 

\section{Theorem Formulation}

We begin this section by introducing a new and improved version of Theorem \ref{yng2} to facilitate the subsequent results. For any time scale $\T$, we have the following result.


\begin{theorem}[Young's Inequality on Time Scales]\label{yng3}
Let $\T$ be any time scale with $\alpha_1\in\T$ and $\sup\T=\infty$. Further, suppose that $f:[\alpha_1,\infty)_{\T}\rightarrow\R$ is a real-valued function satisfying
\begin{enumerate}
 \item $f(\alpha_1)=\beta_1$;
 \item $f$ is continuous on $[\alpha_1,\infty)_\T$, right-dense continuous at $\alpha_1$;
 \item $f$ is strictly increasing on $[\alpha_1,\infty)_\T$ such that $\widetilde{\T}:=f(\T)$ is also a time scale.
\end{enumerate}
Then for any $a\in\left[\alpha_1,\infty\right)_\T$ and $b\in\left[\beta_1,\infty\right)_{\widetilde{\T}}$, we have
\begin{equation}\label{yeq2}
 ab \le \int_{\alpha_1}^a f(t)\Delta t + \int_{\beta_1}^b f^{-1}(y)\widetilde{\nabla} y + \alpha_1\beta_1,
\end{equation}
with equality if and only if $b\in\left\{f^\rho(a),f(a)\right\}$ for fixed $a$, or with equality if and only if $a\in\left\{f^{-1}(b),\sigma(f^{-1}(b))\right\}$ for fixed $b$. The inequality in \eqref{yeq2} is reversed if $f$ is strictly decreasing.
\end{theorem}

\begin{proof}
The proof is modeled after the one given on $\R$ in \cite{diaz}. Note that $f$ is delta integrable and $f^{-1}$ is nabla integrable by the continuity assumption in (ii). For simplicity, define
\begin{equation}\label{Fab}
 F(a,b):=\int_{\alpha_1}^a f(t)\Delta t + \int_{\beta_1}^b f^{-1}(y)\widetilde{\nabla} y - ab + \alpha_1\beta_1.
\end{equation}
Then, the inequality to be shown is just $F(a,b)\ge 0$.

(I). We will first show that
$$ F(a,b)\ge F(a,f(a)) \quad\text{for}\quad a\in[\alpha_1,\infty)_\T \quad\text{and}\quad b\in[\beta_1,\infty)_{\widetilde{\T}}, $$
with equality if and only if $b\in\left\{f^\rho(a),f(a)\right\}$. For any such $a$ and $b$ we have
\begin{equation}\label{differncefs}
 F(a,b) - F(a,f(a)) = \int_{f(a)}^{b}\left[f^{-1}(y)-a\right]\widetilde{\nabla} y.
\end{equation}
Clearly if $b=f(a)$ then the integrals are empty, and if $b=f^\rho(a)$ then
$$ F(a,f^\rho(a)) - F(a,f(a)) = \int^{f(a)}_{f^\rho(a)}\left[a-f^{-1}(y)\right]\widetilde{\nabla} y = [f(a)-f^\rho(a)][a-f^{-1}(f(a))] = 0. $$
Otherwise, since $f^{-1}(y)$ is continuous and strictly increasing for $y\in\widetilde{\T}$, the integrals in \eqref{differncefs} are strictly positive
for $b<f^\rho(a)$ and $b>f(a)$.

(II). We will next show that $F(a,f(a))=F(a,f^\rho(a))=0$. For brevity, put $\varphi(a)=F(a,f(a))$, that is
$$ \varphi(a):=\int_{\alpha_1}^a f(t)\Delta t + \int_{\beta_1}^{f(a)} f^{-1}(y)\widetilde{\nabla} y - af(a) + \alpha_1\beta_1. $$
First, assume $a$ is a right-scattered point. Then
\begin{eqnarray*}
 \varphi^\sigma(a)-\varphi(a) 
 &=& \int_a^{\sigma(a)} f(t)\Delta t + \int_{f(a)}^{f^\sigma(a)} f^{-1}(y)\widetilde{\nabla} y - \sigma(a)f^\sigma(a) + af(a) \\
 &=& [\sigma(a)-a]f(a) + \left[f^\sigma(a)-f(a)\right]f^{-1}(f^\sigma(a)) - \sigma(a)f^\sigma(a) + af(a) = 0.
\end{eqnarray*}
Therefore, if $a$ is a right-scattered point, then $\varphi^\Delta(a)=0$. Next, assume $a$ is a right-dense point. Let $\{a_n\}_{n\in\N}\subset[a,\infty)_\T$ be a decreasing sequence converging to $a$. Then
\begin{eqnarray*}
 \varphi(a_n)-\varphi(a) &=& \int_{a}^{a_n} f(t)\Delta t + \int_{f(a)}^{f(a_n)} f^{-1}(y) \widetilde{\nabla} y  - a_nf(a_n) + af(a) \\
 &\ge& (a_n-a)f(a)+[f(a_n)-f(a)]a - a_nf(a_n) + af(a) \\
 & = & (a_n-a)\left[f(a)-f(a_n)\right], 
\end{eqnarray*}
since the functions $f$ and $f^{-1}$ are strictly increasing.
Similarly, 
\begin{eqnarray*}
 \varphi(a_n)-\varphi(a) 
 &\le& (a_n-a)f(a_n)+[f(a_n)-f(a)]a_n - a_nf(a_n) + af(a) \\
 & = & (a_n-a)\left[f(a_n)-f(a)\right].
\end{eqnarray*}
Therefore, 
$$ 0=\lim_{n\rightarrow\infty}\left[f(a)-f(a_n)\right] \le \lim_{n\rightarrow\infty}\frac{\varphi(a_n)-\varphi(a)}{a_n-a}\le \lim_{n\rightarrow\infty}\left[f(a_n)-f(a)\right]=0. $$
It follows that $\varphi^\Delta(a)$ exists, and $\varphi^\Delta(a)=0$ for right-dense $a$ as well. In other words, in either case, $\varphi^\Delta(a)=0$ for $a\in[\alpha_1,\infty)_\T$. As $\varphi(\alpha_1)=0$, by the uniqueness theorem for initial value problems we have that $\varphi(a)=0$ for all $a\in[\alpha_1,\infty)_\T$. 
From earlier we know that $F(a,f^\rho(a))=F(a,f(a))$. Thus as an overall result, we have that
$$ F(a,b)\ge F(a,f(a))=0, $$
with equality if and only if $b=f(a)$ or $b=f^\rho(a)$, as claimed. The case with $a\in\left\{f^{-1}(b),\sigma(f^{-1}(b))\right\}$ for fixed $b$ is similar and thus omitted. If $f$ is strictly decreasing, it is straightforward to see that the inequality in \eqref{yeq2} is reversed; the details are left to the reader.
\end{proof}

We now focus on establishing an upper bound for Young's integral. Before we state and prove our main theorem, we need an auxiliary result via the following lemma.


\begin{lemma}\label{lemma22}
Let $f$ satisfy the hypotheses of Theorem $\ref{yng3}$, and let $F(a,b)$ be given as in \eqref{Fab}. For every $a,\alpha\in\T$ and $b,\beta\in\widetilde{\T}$ we have
\begin{equation}\label{em4}
 F(a,b)+F(\alpha,\beta) \ge - (\alpha-a)(\beta-b),
\end{equation}
where equality holds if and only if $\alpha\in\left\{f^{-1}(b),\sigma(f^{-1}(b))\right\}$ and $\beta\in\left\{f^\rho(a),f(a)\right\}$.
\end{lemma}

\begin{proof}
Fix $a\in\T$ and $b\in\widetilde{\T}$. By Young's integral inequality on time scales (Theorem \ref{yng3}) we have
\begin{equation}\label{em5}
 \int_{\alpha_1}^a f(t)\Delta t + \int_{\beta_1}^{\beta} f^{-1}(y)\widetilde{\nabla} y + \alpha_1\beta_1 \ge a\beta, \quad\text{and}
\end{equation}
\begin{equation}\label{em6}
 \int_{\alpha_1}^{\alpha} f(t)\Delta t + \int_{\beta_1}^{b} f^{-1}(y)\widetilde{\nabla} y + \alpha_1\beta_1 \ge \alpha b,
\end{equation}
with equality if and only if $\beta\in\left\{f^\rho(a),f(a)\right\}$ and $\alpha\in\left\{f^{-1}(b),\sigma(f^{-1}(b))\right\}$, respectively.
By rearranging it follows that
\begin{eqnarray*}
&& \left(\int_{\alpha_1}^a f(t)\Delta t + \int_{\beta_1}^b f^{-1}(y)\widetilde{\nabla} y - ab + \alpha_1\beta_1\right) 
+ \left(\int_{\alpha_1}^{\alpha} f(t)\Delta t + \int_{\beta_1}^{\beta} f^{-1}(y)\widetilde{\nabla} y - \alpha\beta + \alpha_1\beta_1\right) \\
&& = \left(\int_{\alpha_1}^a f(t)\Delta t + \int_{\beta_1}^{\beta} f^{-1}(y)\widetilde{\nabla} y + \alpha_1\beta_1\right) + \left(\int_{\alpha_1}^{\alpha} f(t)\Delta t + \int_{\beta_1}^{b} f^{-1}(y)\widetilde{\nabla} y + \alpha_1\beta_1\right) - ab - \alpha\beta \\
&& \ge a\beta + \alpha b - ab - \alpha\beta = - \left(\alpha-a\right)\left(\beta-b\right).
\end{eqnarray*}
Note that equality holds here if and only if it holds in \eqref{em5} and \eqref{em6}, videlicet if and only if $\alpha\in\left\{f^{-1}(b),\sigma(f^{-1}(b))\right\}$ and $\beta\in\left\{f^\rho(a),f(a)\right\}$.
\end{proof}


\begin{theorem}\label{newthm1}
Let $\T$ be any time scale and $f:[\alpha_1,\alpha_2]_\T\rightarrow[\beta_1,\beta_2]_{\widetilde{\T}}$ be a continuous strictly increasing function such that $\widetilde{\T}=f(\T)$ is also a time scale. Then for every $a,\widehat{a}\in[\alpha_1,\alpha_2]_\T$ and $b,\widehat{b}\in[\beta_1,\beta_2]_{\widetilde{\T}}$ we have
\begin{equation}\label{em3}
 \left(f^{-1}(\widehat{b})-\widehat{a}\right)\left(f^\rho(\widehat{a})-\widehat{b}\right) 
 \le \int_{\widehat{a}}^a f(t)\Delta t + \int_{\widehat{b}}^b f^{-1}(y)\widetilde{\nabla} y - ab + \widehat{a}\widehat{b} 
 \le \left(f^{-1}(b)-a\right)\left(b-f^\rho(a)\right),
\end{equation}
where the equalities hold if and only if $\widehat{b}\in\left\{f^\rho(\widehat{a}),f(\widehat{a})\right\}$ and $b\in\left\{f^\rho(a),f(a)\right\}$. The inequalities are reversed if $f$ is strictly decreasing.
\end{theorem}

\begin{proof}
Considering $F$ as in \eqref{Fab}, and \eqref{em4} with $\alpha=f^{-1}(b)$ and $\beta=f^\rho(a)$ we have the equality
$$  F(a,b)+F\left(f^{-1}(b),f^\rho(a)\right) = - \left(f^{-1}(b)-a\right)\left(f^\rho(a)-b\right). $$
As $f^{-1}(b)\in[\alpha_1,\alpha_2]_\T$ and $f^\rho(a)\in[\beta_1,\beta_2]_{\widetilde{\T}}$, via Young's integral inequality on time scales (Theorem \ref{yng3} above) we see that $F\left(f^{-1}(b),f^\rho(a)\right) \ge 0$. Consequently we have that
\begin{equation}\label{eq28}
 0 \le F(a,b) \le - \left(f^{-1}(b)-a\right)\left(f^\rho(a)-b\right); 
\end{equation}
note that equality holds if and only if $b\in\left\{f^\rho(a),f(a)\right\}$. Thus for any $\widehat{a}\in[\alpha_1,\alpha_2]_\T$ and $\widehat{b}\in[\beta_1,\beta_2]_{\widetilde{\T}}$ we have from \eqref{eq28} that
\begin{equation}\label{eq29} 
 0 \le - \left(f^{-1}(\widehat{b})-\widehat{a}\right)\left(f^\rho(\widehat{a})-\widehat{b}\right)-F(\widehat{a},\widehat{b}), 
\end{equation}
with equality if and only if $\widehat{b}\in\left\{f^\rho(\widehat{a}),f(\widehat{a})\right\}$.
Combining inequalities \eqref{eq28} and \eqref{eq29} we get
\begin{eqnarray*} 
 0 &\le& F(a,b) - \left(f^{-1}(\widehat{b})-\widehat{a}\right)\left(f^\rho(\widehat{a})-\widehat{b}\right)-F(\widehat{a},\widehat{b}) \\
   &\le& - \left(f^{-1}(b)-a\right)\left(f^\rho(a)-b\right) - \left(f^{-1}(\widehat{b})-\widehat{a}\right)\left(f^\rho(\widehat{a})-\widehat{b}\right)-F(\widehat{a},\widehat{b}).
\end{eqnarray*}
This can be rewritten as \eqref{em3}. If $f$ is strictly decreasing the proof is similar and thus omitted.
\end{proof}


\begin{remark}
In \cite[Theorem 1.1]{em} the author assumes $f(\widehat{a})=\widehat{b}$ $($see \eqref{emorig} above$)$, so that the inequality in \eqref{em3} is new for $\T=\R$, as well as for $\T=\Z$ and general time scales. 
\end{remark}


\begin{remark}
Due to Lemma $\ref{lemma22}$ and the first line of the proof of Theorem $\ref{newthm1}$, there are three other inequalities we could write in place of \eqref{em3}, namely 
\begin{eqnarray*}
  \left(f^{-1}(\widehat{b})-\widehat{a}\right)\left(f(\widehat{a})-\widehat{b}\right) 
 &\le& \int_{\widehat{a}}^a f(t)\Delta t + \int_{\widehat{b}}^b f^{-1}(y)\widetilde{\nabla} y - ab + \widehat{a}\widehat{b} 
 \le \left(f^{-1}(b)-a\right)\left(b-f(a)\right), \\
  \left(\sigma(f^{-1}(\widehat{b}))-\widehat{a}\right)\left(f^\rho(\widehat{a})-\widehat{b}\right) 
 &\le& \int_{\widehat{a}}^a f(t)\Delta t + \int_{\widehat{b}}^b f^{-1}(y)\widetilde{\nabla} y - ab + \widehat{a}\widehat{b} 
 \le \left(\sigma(f^{-1}(b))-a\right)\left(b-f^\rho(a)\right), \\
  \left(\sigma(f^{-1}(\widehat{b}))-\widehat{a}\right)\left(f(\widehat{a})-\widehat{b}\right) 
 &\le& \int_{\widehat{a}}^a f(t)\Delta t + \int_{\widehat{b}}^b f^{-1}(y)\widetilde{\nabla} y - ab + \widehat{a}\widehat{b} 
 \le \left(\sigma(f^{-1}(b))-a\right)\left(b-f(a)\right).
\end{eqnarray*}
If we focus on just the upper bounds, for $a>\widehat{a}$, $b>\widehat{b}$, and $b \ge f(a)$ we have the least of these upper bounds, leading to
$$ \int_{\widehat{a}}^a f(t)\Delta t + \int_{\widehat{b}}^b f^{-1}(y)\widetilde{\nabla} y - ab + \widehat{a}\widehat{b} 
 \le \left(f^{-1}(b)-a\right)\left(b-f(a)\right), $$
whereas for $a>\widehat{a}$, $b>\widehat{b}$, and $b \le f^\rho(a)$ we have
$$ \int_{\widehat{a}}^a f(t)\Delta t + \int_{\widehat{b}}^b f^{-1}(y)\widetilde{\nabla} y - ab + \widehat{a}\widehat{b} 
 \le \left(\sigma(f^{-1}(b))-a\right)\left(b-f^\rho(a)\right). $$
\end{remark}


\begin{cor}\label{cor26}
Pick $a_i\in\T$ with $a_i<a_{i+1}$. Let $f_1:[\rho(a_1),a_{2}]_\T\rightarrow\R$ and $f_i:[a_i,a_{i+1}]_\T\rightarrow\R$ be continuous strictly monotone functions for $i=2,\cdots,m$. Assume $f:\T\rightarrow\R$ is continuous, where 
$$ f(t):=f_i(t) \quad\text{and}\quad f^{-1}(y):=f_i^{-1}(y) $$ 
for $t\in[a_i,a_{i+1}]_\T$ and $y\in[\min\{f(a_i),f(a_{i+1})\},\max\{f(a_i),f(a_{i+1})\}]_{f(\T)}$, respectively, and $i=1,2,\cdots,m$, that is to say $f_i(a_{i+1})=f_{i+1}(a_{i+1})$ for $i=1,2,\cdots,m-1$. Set 
\begin{eqnarray*}
 && I_f:=\int_{a_1}^{a_{m+1}} f(t)\Delta t + \int_{b_1}^{b_{m+1}} f^{-1}(y)\widetilde{\nabla} y - a_{m+1}b_{m+1} + a_1b_1, \quad\text{and} \\
 && K_{i}:=\left(a_{i}-f^{-1}(b_{i})\right)\left(f^\rho(a_{i})-b_{i}\right), \quad i\in\{1,m+1\},
\end{eqnarray*}
where $f^{-1}(b_1)\in[a_1,a_2]_\T$ and $f^{-1}(b_{m+1})\in[a_m,a_{m+1}]_\T$. Then we have the following.
\begin{enumerate}
\item If $f_1$ and $f_m$ are both strictly increasing, then $$-K_{1}\le I_f \le K_{m+1}.$$ The inequalities are reversed if $f_1$ and $f_m$ are both strictly decreasing.
\item If $f_1$ is strictly increasing and $f_m$ is strictly decreasing, then $$K_{m+1}-K_{1} \le I_f \le 0.$$ The inequalities are reversed if $f_1$ is strictly decreasing and $f_m$ is strictly increasing.
\end{enumerate}
In all cases, the equalities hold if and only if $b_1\in\left\{f^\rho(a_1),f(a_1)\right\}$ and $b_{m+1}\in\left\{f^\rho(a_{m+1}),f(a_{m+1})\right\}$. 
\end{cor}

\begin{proof}
We will only prove the first part of (i), as the other parts follow in a similar manner from Theorem \ref{newthm1}. Assume $f_1$ and $f_m$ are both strictly increasing. By Theorem \ref{newthm1} we have the inequalities
$$ \left(f^{-1}(b_1)-a_1\right)\left(f^\rho(a_1)-b_1\right) 
 \le \int_{a_1}^{a_2} f(t)\Delta t + \int_{b_1}^{f(a_2)} f^{-1}(y)\widetilde{\nabla} y - a_2f(a_2) + a_1b_1 \le 0, $$
the equalities
$$ 0 = \int_{a_i}^{a_{i+1}} f(t)\Delta t + \int_{f(a_i)}^{f(a_{i+1})} f^{-1}(y)\widetilde{\nabla} y - a_{i+1}f(a_{i+1}) + a_if(a_i) = 0 $$
for $i=2,\cdots,m-1$, and the inequalities
\begin{eqnarray*}
 0 &\le& \int_{a_m}^{a_{m+1}} f(t)\Delta t + \int_{f(a_m)}^{b_{m+1}} f^{-1}(y)\widetilde{\nabla} y - a_{m+1}b_{m+1} + a_mf(a_m) \\
   &\le& \left(a_{m+1}-f^{-1}(b_{m+1})\right)\left(f^\rho(a_{m+1})-b_{m+1}\right),
\end{eqnarray*}
where equalities hold in the first line if and only if $b_1\in\left\{f^\rho(a_1),f(a_1)\right\}$, and equalities hold in the third if and only if $b_{m+1}\in\left\{f^\rho(a_{m+1}),f(a_{m+1})\right\}$. If we add these expressions together, we obtain $-K_{1}\le I_f \le K_{m+1}$. This completes the proof.
\end{proof}


\begin{remark}
Corollary $\ref{cor26}$ for continuous piecewise-monotone functions is new even for $\T=\R$, as well as for $\T=\Z$ and general time scales. In the next result, Theorem $\ref{pieces}$, we extend the original Young result for continuous functions to piecewise-continuous piecewise-monotone functions on $\R$.
\end{remark}


\begin{theorem}\label{pieces}
Let $a_i\in\R$ with $a_i<a_{i+1}$, and let $f_i:[a_i,a_{i+1}]\rightarrow\R$ be a continuous strictly monotone function for $i=1,2,\cdots,m$. Let $f:[a_1,a_{m+1}]\rightarrow\R$ be the piecewise-continuous function given by 
$$ f(x):=f_i(x)  \quad\text{and}\quad  f^{-1}(y):=f_i^{-1}(y) $$ 
for $x\in(a_i,a_{i+1})$ and $y\in\left(\min\{f(a_i),f(a_{i+1})\},\max\{f(a_i),f(a_{i+1})\}\right)$, respectively, for $i=1,\cdots,m$, with $f(a_1)=f_1(a_1)$ and $f(a_{m+1})=f_m(a_{m+1})$. Set 
\begin{eqnarray*}
 && F(b_1,b_{m+1}):=\int_{a_1}^{a_{m+1}} f(x)dx + \int_{b_1}^{b_{m+1}} f^{-1}(y)dy - a_{m+1}b_{m+1} + a_1b_1+\sum_{i=2}^{m}a_i\left[f_i(a_i)-f_{i-1}(a_i)\right],  \\
 && K_{i}:=\left(a_{i}-f^{-1}(b_{i})\right)\left(f(a_{i})-b_{i}\right), \quad i\in\{1,m+1\},
\end{eqnarray*}
where $f^{-1}(b_1)\in[a_1,a_2]$ and $f^{-1}(b_{m+1})\in[a_m,a_{m+1}]$. Then we have the following.
\begin{enumerate}
\item If $f_1$ and $f_m$ are both strictly increasing, then $$-K_{1}\le F(b_1,b_{m+1}) \le K_{m+1}.$$ The inequalities are reversed if $f_1$ and $f_m$ are both strictly decreasing.
\item If $f_1$ is strictly increasing and $f_m$ is strictly decreasing, then $$K_{m+1}-K_{1} \le F(b_1,b_{m+1}) \le 0.$$ The inequalities are reversed if $f_1$ is strictly decreasing and $f_m$ is strictly increasing.
\end{enumerate}
In all cases, the equalities hold if and only if $b_1=f(a_1)$ and $b_{m+1}=f(a_{m+1})$. 
\end{theorem}

\begin{proof}
Apply Theorem \ref{newthm1} on $\T=\R$ to the pieces $f_i$ on $[a_i,a_{i+1}]$, using the appropriate inequalities for $f_1$ and $f_m$, and equalities for $f_i$, $i=2,\cdots,m-1$. Then add up these expressions to get the result.
\end{proof}


\begin{theorem}\label{legendrethm}
For some delta differentiable function $g$, let $f=g^\Delta$ satisfy the hypotheses of Theorem $\ref{newthm1}$, where $f^{-1}=(g^\Delta)^{-1}=(g^*)^{\widetilde{\nabla}}$ on $\widetilde{\T}=f(\T)$. [If $\T=\R$ the function $g^*$ is known as the Legendre transform of $g$.] If we pick $\widehat{a}\widehat{b}=g(\widehat{a})+g^*(\widehat{b})$, then by Theorem $\ref{newthm1}$ we have
\begin{eqnarray*}
 \left((g^*)^{\widetilde{\nabla}}(\widehat{b})-\widehat{a}\right)\left(g^\nabla(\widehat{a})-\widehat{b}\right) 
 \le g(a) + g^*(b) - ab 
 \le \left((g^*)^{\widetilde{\nabla}}(b)-a\right)\left(b-g^\nabla(a)\right),
\end{eqnarray*}
where the equalities hold if and only if $\widehat{b}\in\left\{g^\nabla(\widehat{a}),g^\Delta(\widehat{a})\right\}$ and $b\in\left\{g^\nabla(a),g^\Delta(a)\right\}$.
\end{theorem}

In the following theorem we reconsider Theorem \ref{newthm1} above. This allows us to get a Young-type integral inequality without having to find $f^{-1}$.


\begin{theorem}\label{newthm2}
Let the hypotheses of Theorem $\ref{newthm1}$ hold. Then for any $a,\alpha,\widehat{a},\widehat{\alpha}\in[\alpha_1,\alpha_2]_\T$ we have
\begin{eqnarray}\label{yeq3}
 \left(\widehat{\alpha}-\widehat{a}\right)\left(f^\rho(\widehat{a})-f(\widehat{\alpha})\right) 
 &\le& \int_{\widehat{a}}^a f(t)\Delta t - \int_{\widehat{\alpha}}^{\alpha} f(t)\Delta t + (\alpha-a)f(\alpha) + (\widehat{a}-\widehat{\alpha})f(\widehat{\alpha}) \nonumber  \\
 &\le& \left(\alpha-a\right)\left(f(\alpha)-f^\rho(a)\right),
\end{eqnarray}
where the equalities hold if and only if $\widehat{\alpha}\in\{\rho(\widehat{a}),\widehat{a}\}$ and $\alpha\in\{\rho(a),a\}$.
\end{theorem}

\begin{proof}
By Theorem \ref{newthm1} with $\widehat{a}=\widehat{\alpha}$, $\widehat{b}=f(\widehat{\alpha})$, $a=\alpha$ and $b=f(\alpha)$ we have 
\begin{equation}\label{31pf3}
 \int_{f(\widehat{\alpha})}^{f(\alpha)} f^{-1}(y)\widetilde{\nabla} y = \alpha f(\alpha) - \widehat{\alpha}f(\widehat{\alpha}) -  \int_{\widehat{\alpha}}^{\alpha} f(t)\Delta t
\end{equation}
for any $\alpha,\widehat{\alpha}\in[\alpha_1,\alpha_2]_\T$. Since $\alpha,\widehat{\alpha}\in[\alpha_1,\alpha_2]_\T$ are arbitrary, we substitute \eqref{31pf3} into \eqref{em3} to obtain \eqref{yeq3}.
\end{proof}


\begin{example}
Consider the generalized polynomial functions $h_n(t,s)$ on time scales defined recursively in the following way \cite[Section 1.6]{bp1}:
$$ h_0(t,s)\equiv 1, \quad h_{k+1}(t,s)=\int_s^t h_k(\tau,s)\Delta\tau, \quad t,s\in\T, \quad k\in\N_0. $$
If we take $f(t)=h_n(t,\widehat{\alpha})$ for any $n\in\N$, then by Theorem $\ref{newthm2}$ we have
\begin{eqnarray}\label{exz2}
 \left(\widehat{\alpha}-\widehat{a}\right)h_n(\rho(\widehat{a}),\widehat{\alpha}) 
 &\le& h_{n+1}(a,\widehat{a}) - h_{n+1}(\alpha,\widehat{\alpha}) + (\alpha- a)h_n(\alpha,\widehat{\alpha}) \nonumber  \\
 &\le& \left(\alpha-a\right)\left[h_n(\alpha,\widehat{\alpha})-h_n(\rho(a),\widehat{\alpha})\right],
\end{eqnarray}
for any $a,\widehat{a},\alpha,\widehat{\alpha}\in[\alpha_1,\alpha_2]_\T$, where the equalities hold if and only if $\widehat{\alpha}\in\{\rho(\widehat{a}),\widehat{a}\}$ and $\alpha\in\{\rho(a),a\}$.
\end{example}


\section{results for difference equations}

In this section we concentrate on the discrete case. For $\T=\Z$ we have the following new discrete results, which are corollaries of the theorems above. Recall that $[\alpha_1,\alpha_2]_\Z=\{\alpha_1,\alpha_1+1,\alpha_1+2,\cdots,\alpha_2-1,\alpha_2\}$. The first two theorems are direct translations to $\T=\Z$ of Theorem \ref{newthm1} and Theorem \ref{newthm2}, respectively.


\begin{theorem}\label{newthm1Z}
Let $f:[\alpha_1,\alpha_2]_\Z\rightarrow[\beta_1,\beta_2]_{\widetilde{\Z}}$ be a strictly increasing function, where $\widetilde{\Z}=f(\Z)$. Then for every $a,\widehat{a}\in[\alpha_1,\alpha_2]_\Z$ and $b,\widehat{b}\in[\beta_1,\beta_2]_{\widetilde{\Z}}$ we have
\begin{eqnarray}\label{em3z}
 \left(f^{-1}(\widehat{b})-\widehat{a}\right)\left(f(\widehat{a}-1)-\widehat{b}\right) 
 &\le& \sum_{t=\widehat{a}}^{a-1} f(t) + \sum_{y\in(\widehat{b},b]\cap\widetilde{\Z}} f^{-1}(y)\widetilde{\mu}(y) - ab + \widehat{a}\widehat{b} \nonumber \\
 &\le& \left(f^{-1}(b)-a\right)\left(b-f(a-1)\right),
\end{eqnarray}
where the equalities hold if and only if $\widehat{b}\in\left\{f(\widehat{a}-1),f(\widehat{a})\right\}$ and $b\in\left\{f(a-1),f(a)\right\}$.
\end{theorem}


\begin{theorem}\label{newthm2Z}
Let $f:\Z\rightarrow\R$ be a strictly increasing function. Then for any integers $a,\widehat{a},\alpha,\widehat{\alpha}$ we have
\begin{eqnarray}\label{yeq3z}
 \left(\widehat{\alpha}-\widehat{a}\right)\left[f(\widehat{a}-1)-f(\widehat{\alpha})\right] 
 &\le& \sum_{t=\widehat{a}}^{a-1} f(t) - \sum_{t=\widehat{\alpha}}^{\alpha-1} f(t) + (\alpha-a)f(\alpha) + (\widehat{a}-\widehat{\alpha})f(\widehat{\alpha}) \nonumber  \\
 &\le& \left(\alpha-a\right)\left[f(\alpha)-f(a-1)\right],
\end{eqnarray}
where the equalities hold if and only if $\widehat{\alpha}\in\{\widehat{a}-1,\widehat{a}\}$ and $\alpha\in\{a-1,a\}$.
\end{theorem}


\begin{example}
Consider the Pochhammer lower factorial function 
$$ f_k(t)=t^{\underline{k}}:=t(t-1)\cdots(t-k+1), \quad t,k\in\Z, $$ 
also known as $t$ to the $k$ falling \cite{kp}, or the falling factorial power function \cite{vl}. It is clear that $f_k$ is strictly increasing on the integer interval $[k-1,\infty)_\Z$. By Theorem $\ref{newthm2Z}$ we have
\begin{eqnarray}\label{exz1}
 (a-\alpha)f_k(\alpha) \le \frac{1}{k+1}\left[f_{k+1}(a)-f_{k+1}(\alpha)\right] \le (a-\alpha)f_k(a-1)
\end{eqnarray}
for $a,\alpha\in\{k-1,k,k+1,\cdots\}$, where the equalities hold if and only if $\alpha\in\{a-1,a\}$.
\end{example}


\begin{example}
For any real $B>1$ and any integers $a\ge\alpha\in\Z$ we have
$$ B^{\alpha} \le \frac{B^a-B^{\alpha}}{(a-\alpha)(B-1)} \le B^{a-1}, $$
where the equalities hold if and only if $\alpha\in\{a-1,a\}$.
\end{example}


\begin{example}
Consider Theorem $\ref{legendrethm}$. Let $\T=\Z$ and $f(t)=B^t$ for $B>1$. Then $\widetilde{\T}=B^{\Z}$, and we have $g(t)=\frac{B^t-B^{\alpha}}{B-1}$, $f^{-1}(y)=\log_{B}y$, and $g^*(y)=y\log_{B}y-\frac{y}{B-1}+\beta\left(\alpha+\frac{1}{B-1}-\log_{B}\beta\right)$. It is easy to check that
$f^{-1}=(g^*)^{\widetilde{\nabla}}$ on $\widetilde{\T}$ and $\alpha\beta=g(\alpha)+g^*(\beta)$. Therefore we have the inequalities
\begin{eqnarray*}
 \left(\log_{B} \beta-\alpha\right)\left(B^{\alpha-1}-\beta\right) 
 &\le& \frac{B^a-B^{\alpha}}{B-1} + b\log_{B}b-\frac{b}{B-1}+\beta\left(\alpha+\frac{1}{B-1}-\log_{B}\beta\right) - ab \\
 &\le& \left(\log_{B}b-a\right)\left(b-B^{a-1}\right),
\end{eqnarray*}
where the equalities hold if and only if $\beta\in\left\{B^{\alpha-1},B^{\alpha}\right\}$ and $b\in\left\{B^{a-1},B^a\right\}$.
\end{example}


\begin{example}
Let $f(t)=\sin\left[\frac{\pi t}{2k}\right]$ for $k\in\N$. Then $f$ is strictly increasing on $[-k,k]_{\Z}$, so that for any $a\ge\alpha\in[-k,k]_{\Z}$ we have by Theorem $\ref{newthm2Z}$ that
$$ \sin\left[\frac{\alpha\pi}{2k}\right] \le \frac{1}{2(a-\alpha)}\left(\cos\left[\frac{(2\alpha-1)\pi}{4k}\right] - \cos\left[\frac{(2a-1)\pi}{4k}\right]\right)\csc\left[\frac{\pi}{4k}\right] \le \sin\left[\frac{(a-1)\pi}{2k}\right], $$
with equalities if and only if $\alpha\in\{a-1,a\}$.
\end{example}


\begin{example}
Let $f(t)=\binom{t}{k}$ for $t,k\in\N$ with $t\ge k$. Then $f$ is strictly increasing on $[k,\infty)_{\Z}$, whereby for any $a\ge\alpha\in[k,\infty)_{\Z}$ we have
$$ \binom{\alpha}{k} \le \frac{(a-k)\binom{a}{k}-(\alpha-k)\binom{\alpha}{k}}{(a-\alpha)(k+1)} \le \binom{a-1}{k}, $$
with equalities if and only if $\alpha\in\{a-1,a\}$.
\end{example}



\begin{thebibliography}{99}

\bibitem{and} 
D. R. Anderson,
Young's integral inequality on time scales revisited,
\emph{J. Inequalities in Pure Appl. Math.}, Volume 8 (2007), Issue 3, Article 64, 5 pages. 

\bibitem{bp1} 
M. Bohner and A. Peterson,
\emph{Dynamic Equations on Time Scales: An Introduction with Applications},
Birkh\"auser, Boston (2001).

\bibitem{diaz} 
J. B. Diaz and F. T. Metcalf,
An analytic proof of Young's inequality,
\emph{American Math. Monthly}, 77 (1970) 603--609.
 
\bibitem{hardy} 
G. H. Hardy, J. E. Littlewood, and G. P\'{o}lya,
\emph{Inequalities}, Cambridge, 1934.
  
\bibitem{hilger} 
S. Hilger,
Analysis on measure chains---a unified approach to continuous and discrete calculus, 
\emph{Results Math.} 18 (1990), no. 1--2, 18--56.

\bibitem{kp} 
W. G. Kelley and A. C. Peterson,
\emph{Difference Equations: An Introduction with Applications}, 2e,	Harcourt/Academic Press, San Diego, 2001.
	
\bibitem{vl} 
V. Lampret,
Approximating real Pochhammer products: a comparison with powers,
\emph{Cent. Eur. J. Math.}, 7(3) (2009) 493--505.

\bibitem{merkle} 
M. Merkle, 
A contribution to Young's inequality, 
\emph{Univ. Beograd. Publ. Elektrotehn. Fak. Ser. Mat. Fiz.}, 461--497 (1974), 265--267.

\bibitem{em}
E. Minguzzi,
An equivalent form of Young's inequality with upper bound, 
\emph{Appl. Anal. Discrete Math.} 2 (2008) 213--216.

\bibitem{fink} 
D. S. Mitrinovi\'{c}, J. E. Pe\v{c}ari\'{c}, and A. M. Fink, 
\emph{Classical and New Inequalities in Analysis}, Kluwer Academic Publishers, Dordrecht, 1993.
 
\bibitem{tolsted}
E. Tolsted, 
An elementary derivation of the Cauchy, Holder, and Minkowski inequalities from Young's inequality, 
\emph{Math. Mag.}, 37 (1964), 2--12.

\bibitem{wyyh}
F. H. Wong, C. C. Yeh, S. L. Yu, and C. H. Hong,
Young's inequality and related results on time scales,
\emph{Appl. Math. Letters}, 18 (2005) 983--988.

\bibitem{young}
W. H. Young, 
On classes of summable functions and their Fourier series,
\emph{Proc. Royal Soc.}, Series (A), 87 (1912) 225--229.
\end{thebibliography}
\end{document}